\documentclass[
               twoside,
               ]{amsart}
\usepackage{amssymb} % additional symbols
\usepackage{MnSymbol} % more arrows and symbols
\usepackage{mathtools} % enhances amsmath
\usepackage[all]{xy} % commutative diagrams
\usepackage{color} % coloured text
\usepackage{ifdraft} % draft mode checking
\usepackage{lastpage} % determination of last page
\usepackage[pdfstartview=FitH]{hyperref} % hyperreferences

\DeclareMathAlphabet{\mathpzc}{OT1}{pzc}{m}{it}

\newtheoremstyle{nametheorem}% name
  {3pt}%      Space above, empty = `usual value'
  {3pt}%      Space below
  {\itshape}% Body font
  {}%         Indent amount (empty = no indent, \parindent = para indent)
  {\bfseries}% Thm head font
  {.}%        Punctuation after thm head
  {.5em}%     Space after thm head: " " = normal interword space;
  {\thmname{#1} \thmnote{#3}}% Thm head spec

\newtheorem{thm}{Theorem}[section]
\newtheorem{cor}[thm]{Corollary}
\newtheorem{lem}[thm]{Lemma}
\newtheorem{prop}[thm]{Proposition}

\theoremstyle{nametheorem}
\newtheorem*{conj}{Conjecture}
\theoremstyle{definition}
\newtheorem{defn}[thm]{Definition}

\theoremstyle{remark}
\newtheorem{rem}[thm]{Remark}

\numberwithin{equation}{section}
\DeclarePairedDelimiterX\set[1]\{\}{%

#1
}

\newcommand*{\mto}{\rightarrow} % maps
\newcommand*{\qto}{\twoheadrightarrow} % quotient maps
\newcommand*{\Int}{\mathbb{Z}} % integers
\newcommand*{\Rat}{\mathbb{Q}} % rationals
\newcommand*{\N}{\mathbb{N}} % natural numbers
\newcommand*{\FF}{\mathbb{F}} % finite field
\newcommand*{\cmplx}[1]{{#1}^\bullet} % complex
\newcommand*{\tensor}{\otimes} % tensor product
\newcommand*{\Ltensor}{\tensor^{\mathbb{L}}} % derived tensor product
\newcommand*{\isomorph}{\cong} % Isomorphie
\newcommand*{\et}{\mathrm{\acute{e}t}} % etale
\newcommand*{\algc}[1]{\overline{#1}} % separable closure
\newcommand{\nr}{\mathrm{nr}} % nonramified closure
\newcommand{\ab}{\mathrm{ab}} % abelinisation
\newcommand{\cyc}{\mathrm{cyc}} % cyclotomic
\newcommand*{\pdual}[1]{{#1}^{\vee}} % Pontryagin dual
\newcommand{\IntR}{\mathcal{O}} % ring of integers
\newcommand{\cs}{\mathrm{cs}} %completely decomposed
\newcommand{\sesi}{\mathrm{ss}} % semisimple
\newcommand{\Jac}{\mathfrak{J}} % Jacobson radical
\newcommand{\Sub}{\mathit{Sub}} % submodule conjecture

\DeclareMathOperator{\HF}{H} % cohomology
\DeclareMathOperator{\Hom}{Hom} % morphisms
\DeclareMathOperator{\Ext}{Ext} % extension functor
\DeclareMathOperator{\End}{End} % endomorphisms
\DeclareMathOperator{\im}{im} % image
\DeclareMathOperator{\coker}{coker} % cokernel
\DeclareMathOperator{\Spec}{Spec} % spectrum
\DeclareMathOperator{\Gal}{G} % Galois group
\DeclareMathOperator{\cd}{cd} % cohomological dimension
\DeclareMathOperator{\Cl}{Cl} % class group
\DeclareMathOperator{\Tate}{T} % Tate module
\DeclareMathOperator{\rk}{rk} % rank
\newdir{ >}{{}*!/-5pt/@{>}} %arrow tail
\newdir{O}{{}*-\cir<2.5pt>{}} % big circle
\xyoption{2cell} % 2-cell feature
\UseTwocells %
\xyoption{poly}
\allowdisplaybreaks[1]

\title{Fine Selmer Groups and Isogeny Invariance}
\author{R.~Sujatha}
\author{M.~Witte}

\address{R.~Sujatha\newline Department of Mathematics, \newline University of British Columbia, \newline Vancouver, \newline Canada V6T1Z2}%
\email{sujatha@math.ubc.ca}%

\address{M.~Witte\newline Universit\"at Heidelberg,\newline
Mathematisches Institut,\newline
Im Neuenheimer Feld 205,\newline
D-69120 Heidelberg }%
\email{witte@mathi.uni-heidelberg.de}%


\date{\today}

\begin{document}

\maketitle
\begin{abstract}
We investigate fine Selmer groups for elliptic curves and for Galois representations over a number field. More specifically, we discuss Conjecture A, which states that the fine Selmer group of an elliptic curve over the cyclotomic extension is a finitely generated $\Int_p$-module. The relationship between this conjecture and Iwasawa's classical $\mu=0$ conjecture is clarified. We also present some partial results towards the question whether Conjecture A is invariant under isogenies.
\end{abstract}

\section{Introduction}

The fine Selmer group is a subgroup of the Selmer group of elliptic curves which plays an important role in Iwasawa theory. More generally, it can be defined for any Galois representation over a number field. It has been widely studied by many authors \cite{Billot:quelques}, \cite[App.~B]{PerrinRiou:padicLFunctions}, \cite{Greenberg:ProjectiveModules}, \cite{Wuthrich:IwTHfineSelmer}, \cite{CoatesSujatha:FineSelmer}. In the last reference, the authors conjecture  that the fine Selmer group of an elliptic curve over the cyclotomic extension is a finitely generated $\Int_p$-module (Conjecture A in \emph{loc.~cit.}). There is ample numerical evidence in support of this conjecture \cite{Wuthrich:IwTHfineSelmer}, \cite{Aribam:OntheMuInvariant}. It is further stated (rather optimistically) in \cite{CoatesSujatha:FineSelmer} that Conjecture A should be invariant under isogeny. The results in this article indicate that the extent of optimism expressed towards isogeny invariance is not commensurate with what one can concretely prove. The relationship between Conjecture A and Iwasawa's classical $\mu=0$ conjecture is already touched upon in \cite{CoatesSujatha:FineSelmer} and \cite{Greenberg:ProjectiveModules}. Our results in this paper demonstrate the depth of this relationship. In particular, the results seem to indicate that even the seemingly weaker isogeny invariance conjecture is potentially as difficult as the $\mu=0$ conjecture.

This article has six sections. In Section~\ref{sec:Notation}, we introduce notation that is used subsequently throughout the paper. In Section~\ref{sec:Fine Selmer}, we discuss the fine Selmer group and its relation to other cohomological modules. Several assertions that are equivalent to Conjecture A are discussed in Section~\ref{sec:Conjecture A}. In Section~\ref{sec:IsogenyInvariance}, we investigate the isogeny invariance of Conjecture A for elliptic curves. In particular, we prove isogeny invariance for a large class of CM elliptic curves. In the final Section~\ref{sec:ClosingRemarks}, we point towards a connection of isogeny invariance and cup products in cohomology.

We would like to thank Karl Rubin, Christian Wuthrich, and Benedict Gross for helpful conversations on CM elliptic curves. We also thank Zheng Li for helpful insights into this problem.

\section{Notation}\label{sec:Notation}

Throughout the text, we will use the following notation. For any field $F$, we let $\Gal_F$ denote its absolute Galois group. The letter $p$ will always denote an odd prime. If $K$ is a number field, a \emph{prime} of $K$ will always refer to a non-archimedean place of $K$. If $S$ is a finite set of primes of $K$, we will write $K_S$ for the maximal extension of $K$ which is unramified outside $S$ and the archimedean places of $K$. We will always assume that $S$ contains all primes of $K$ above $p$. For any subextension $L/K$ of $K_S/K$ we let $\Gal_S(L)$ denote the Galois group of $K_S/L$ and $S_L$ the primes of $L$ above $S$. If $v$ is a prime of $L$, we write $L_v$ for the completion of $L$ at $v$ and $k(v)$ for its residue field. For any profinite group $G$, we write $\cd_p G$ for its $p$-cohomological dimension. For any compact or discrete $G$-module $M$, we write $\HF^i(G,M)$ for the $i$-th continuous cohomology group of $M$ and $\pdual{M}$ for its Pontryagin dual. For any compact or discrete $\Gal_S(K)$-module $M$ and any integer $n$, we write $M(n)$ for its $n$-th Tate twist.

We write $K_\cyc$ for the cyclotomic $\Int_p$-extension of $K$ and let $\Gamma$ denote the Galois group of $K_\cyc/K$. The profinite group rings of $\Gamma$ over $\Int_p$ and $\FF_p$ will be denoted by
\[
 \Lambda=\Int_p[[\Gamma]],\qquad\Omega=\FF_p[[\Gamma]],
\]
respectively. We will write $Q$ for the quotient field of $\Omega$.

More generally, let $\mathcal{L}/K$ be any Galois subextension of $K_S/K$ such that for some finite extension $K'/K$ inside $\mathcal{L}$, the extension $\mathcal{L}/K'$ is pro-$p$. In particular, since the maximal pro-$p$-quotient of $G_S(K')$ is topologically finitely generated \cite[Thm. 10.7.12, Prop. 3.9.1]{NSW:CohomNumFields}, the Galois group $\mathcal{G}=\Gal(\mathcal{L}/K)$ is topologically finitely generated, as well. The profinite group ring $\Int_p[[\mathcal{G}]]$ is then an \emph{adic ring} in the sense that it is compact for the topology defined by the powers of its Jacobson radical $\Jac_{\Int_p[[\mathcal{G}]]}$ \cite[Prop. 3.2]{Witte:MCVarFF}. We may  equip $\Int_p[[\mathcal{G}]]$ with a continuous $\Gal_S(K)$-action by letting $g\in\Gal_S(K)$ act as multiplication by the inverse of its image in $\mathcal{G}$. Let $\Int_p[[\mathcal{G}]]^\sharp$ denote the corresponding $\Gal_S(K)$-module.

Suppose that $M$ is a compact $\Int_p$-module with a continuous $\Gal_S(K)$-action such that $\pdual{M}$ is a countable set. In other words, $M$ has a presentation
\[
M=\varprojlim_{n\in\N}M_n
\]
with $M_n$ finite. We set
\begin{equation}\label{eqn:def of coinduction}
M_{\mathcal{L}}=\varprojlim_{n\in\N}\Int_p[[\mathcal{G}]]^\sharp/\Jac_{\Int_p[[\mathcal{G}]]}^n\tensor_{\Int_p}M_n
\end{equation}
so that $M_{\mathcal{L}}$ is the completed tensor product over $\Int_p$ of $\Int_p[[\mathcal{G}]]^\sharp$ and $M$.

Let $v$ be a prime of $K$. We then have
\begin{equation}\label{eqn:Iwasawa cohomology}
\begin{aligned}
 \HF^i(\Gal_S(K),M_{\mathcal{L}})&=\varprojlim_{K\subset L\subset \mathcal{L}}\HF^i(\Gal_S(L),M),\\
 \HF^i(\Gal_{K_v},M_{\mathcal{L}})&=\varprojlim_{K\subset L\subset \mathcal{L}}\bigoplus_{w\mid v}\HF^i(G_{L_w},M)
\end{aligned}
\end{equation}
where the limit is taken over all finite subextensions $L/K$ of $\mathcal{L}/K$ with respect to the corestriction map \cite[Cor. 2.7.6]{NSW:CohomNumFields}. In particular, $\HF^i(\Gal_S(K),M_{\mathcal{L}})$ agrees with the group denoted by $\mathcal{Z}^i(M/\mathcal{L})$ in \cite{CoatesSujatha:FineSelmer}.

For any integral domain $R$ with quotient field $Q(R)$ and any finitely generated $R$-module $T$, we will write $\rk_R T$ for the dimension of the vector space $Q(R)\tensor_R T$ over $Q(R)$. If $M$ is a  $\Gal_K$-representation on a finite-dimensional vector space over $\FF_p$, we set
\[
 \begin{aligned}
 r_M&=\rk_{\FF_p}\bigoplus_{v\mid\infty}\HF^0(K_v,\pdual{M}(1))\\
    &=-\rk_{\FF_p}\HF^0(G_S(K),M)+\rk_{\FF_p}\HF^1(G_S(K),M)-\rk_{\FF_p}\HF^2(G_S(K),M),
 \end{aligned}
\]
where $S$ is any finite set of  primes of $K$ containing the primes above $p$ such that $M$ is unramified outside $S$. In other words, $r_M$ is the negative of the Euler characteristic of $M$ as considered in \cite[8.7.4]{NSW:CohomNumFields}. We will write $M^{\sesi}$ for the semisimplification of $M$.

For any elliptic curve $E$ over $K$, we let $E[n]$ denote its $n$-torsion points considered as a $\Gal_K$-module. We write
\[
\Tate_p E=\varprojlim_n E[p^n]
\]
for the $p$-adic Tate module and
\[
 E[p^\infty]=\varinjlim_n E[p^n]
\]
for the $p$-power division points of $E$.

\section{The Fine Selmer Group and some Useful Exact Sequences}\label{sec:Fine Selmer}

Let us recall the definition of the fine Selmer group from \cite{CoatesSujatha:FineSelmer}. Let $K$ be a number field, $p$ an odd prime number, and $S$ a finite set of primes of $K$ containing the primes above $p$. Assume that $\mathcal{L}/K$ is any subextension of $K_S/K$.

\begin{defn}
For any discrete $\Gal_S(\mathcal{L})$-module $M$, the \emph{fine Selmer group} of $M$ is given by
\[
R_S(M/\mathcal{L})=\varinjlim_{L/K}\ker\left(\HF^1(\Gal_S(L),M)\mto \bigoplus_{v\in S_{L}}\HF^1(\Gal_{L_{v}},M)\right),
\]
with $L/K$ running through the finite subextensions of $\mathcal{L}/K$. We write
\[
Y_S(M/\mathcal{L})=\pdual{R_S(M/\mathcal{L})}
\]
for the Pontryagin dual of the fine Selmer group of $M$.
\end{defn}

In general, $R_S(M/\mathcal{L})$ and $Y_S(M/\mathcal{L})$ do depend on the choice of $S$. However, if we assume that $\mathcal{L}$ contains the cyclotomic $\Int_p$-extension $K_\cyc$, then the groups do not change if we enlarge $S$. In this case, we drop the $S$ from the notation.

The justification is given as follows. As
\[
R_S(M/\mathcal{L})=\varinjlim_{L/K}R_S(M/L_\cyc),
\]
with $L/K$ running through the finite subextensions of $\mathcal{L}/K$, we may as well assume that $\mathcal{L}=K_\cyc$.

Let $v$ be a  prime of $K_\cyc$. Write $I_v\subset \Gal_{(K_\cyc)_v}$ for the inertia group of $v$ and let $k(v)$ be the residue field. The Hochschild-Serre spectral sequence gives an exact sequence
\begin{equation}\label{eqn:HS for inertia}
0\mto\HF^1(\Gal_{k(v)},M^{I_v})\mto\HF^1(\Gal_{(K_\cyc)_v},M)\mto\HF^0(\Gal_{k(v)},\HF^1(I_v,M))\mto 0.
\end{equation}
and an isomorphism
\[
\HF^2(\Gal_{(K_\cyc)_v},M)\isomorph\HF^1(\Gal_{k(v)},\HF^1(I_v,M)).
\]
For any prime $v$ of $K_\cyc$ not dividing $p$, $\Gal_{k(v)}$ is of order prime to $p$ and the exact sequence \eqref{eqn:HS for inertia} reduces to an isomorphism
\begin{equation}\label{eqn:local cohomology for the cyclotomic extension}
\HF^1(\Gal_{(K_\cyc)_v},M)\isomorph\HF^0(\Gal_{k(v)},\HF^1(I_v,M)),
\end{equation}
while
\[
\HF^1(\Gal_{k(v)},\HF^1(I_v,M))=0.
\]
Furthermore, if $v$ is not contained in $S_{K_\cyc}$, then
\[
\HF^1(I_v,M)\isomorph M(-1)
\]
as $\Gal_{k(v)}$-modules.

If $T$ is any finite set of primes of $K$ containing $S$, we thus obtain a long exact Gysin sequence
\begin{equation}\label{eqn:Gysin sequence}
\begin{split}
0&\mto\HF^1(\Gal_S(K_\cyc),M)\mto\HF^1(\Gal_T(K_\cyc),M)\mto\bigoplus_{v\in (T-S)_{K_\cyc}}\HF^0(\Gal_{k(v)},M(-1))\\
 &\mto\HF^2(\Gal_S(K_\cyc),M)\mto\HF^2(\Gal_T(K_\cyc),M)\mto 0.
\end{split}
\end{equation}
Note that every prime of $K$ splits into only finitely many primes of $K_\cyc$, so that $(T-S)_{K_\cyc}$ is still a finite set.

Hence,
\begin{equation}\label{eqn:idependence of S of the fine Selmer group}
\begin{split}
&\qquad R_T(M/K_\cyc)\\
&=\ker\left(\HF^1(\Gal_T(K_\cyc),M)\mto\smashoperator[r]{\bigoplus_{v\in(T-S)_{K_\cyc}}}\HF^0(\Gal_{k(v)},M(-1))\oplus\smashoperator{\bigoplus_{v\in S_{K_\cyc}}}\HF^1(\Gal_{(K_\cyc)_v},M)\right)\\
&=\ker\left(\HF^1(\Gal_S(K_\cyc),M)\mto\bigoplus_{v\in S_{K_\cyc}}\HF^1(\Gal_{(K_\cyc)_v},M)\right)\\
&=R_S(M/K_\cyc).
\end{split}
\end{equation}

We also obtain isomorphisms
\[
\begin{aligned}
R(M/K_\cyc)&\isomorph\ker\left(\HF^1(\Gal_{K_\cyc},M)\mto\bigoplus_{v}\HF^1(\Gal_{(K_\cyc)_v},M)\right)\\
&\isomorph\ker\left(\HF^1(\Gal_{K_\cyc},M)\mto\bigoplus_{v\nmid p}\HF^1(I_v,M)\oplus\bigoplus_{v\mid p}\HF^1(\Gal_{K_\cyc},M)\right),
\end{aligned}
\]
with $v$ running through all  primes of $K_\cyc$.
%In particular, $R(M/K_\cyc)$ agrees with Greenberg's Selmer group $\mathcal{S}(M/K_\cyc,(\emptyset)_{v\mid p})$ \comment{(Check notation!)} defined in %\comment{[Add reference]}.
For the first isomorphism, we note that
\[
\Gal_{K_\cyc}=\varprojlim_{T}\Gal_T(K_\cyc),
\]
where $T$ runs through all finite sets of  primes of $K$ containing $S$ and then pass to the direct limit over all such $T$ in equation~\eqref{eqn:idependence of S of the fine Selmer group}. For the second isomorphism, we note from \eqref{eqn:local cohomology for the cyclotomic extension} that
\[
\HF^1(\Gal_{(K_\cyc)_v},M)\mto\HF^1(I_v,M)
\]
is injective for all  primes $v$ of $K_\cyc$ not lying over $p$.

We digress briefly to express $R(M/K_\cyc)$ in terms of \'etale cohomology as follows. Consider $M$ as an \'etale sheaf on $\Spec K_\cyc$. Write $\mathcal{O}_{\cyc}$ for its ring of integers and
\[
\eta\colon\Spec K_\cyc\mto\Spec \mathcal{O}_{\cyc}
\]
for the inclusion of the generic point. Then it follows easily from \cite[Prop.~II.2.9]{Milne:ADT} that
\[
\HF^1_{\et}(\Spec \mathcal{O}_{\cyc},\eta_*M)\isomorph\ker\left(\HF^1(\Gal_S(K_\cyc),M)\mto\bigoplus_{v\in S_{K_\cyc}}\HF^0(\Gal_{k(v)},\HF^1(I_v,M))\right)
\]
and
\[
\begin{split}
R(M/K_\cyc)&\isomorph\im\left(\HF^1_{c}(\Spec \mathcal{O}_{\cyc}[\frac{1}{p}],\eta_*M)\mto\HF^1_{\et}(\Spec \mathcal{O}_{\cyc},\eta_*M)\right)\\
&\isomorph\ker\left(\HF^1_{\et}(\Spec \mathcal{O}_{\cyc},\eta_*M)\mto\bigoplus_{v\mid p}\HF^1(\Gal_{k(v)},M^{I_v})\right).
\end{split}
\]
Under the assumption that $\HF^0(\Gal_{k(v)},M^{I_v})=0$ for all primes $v$ of $K_\cyc$ dividing $p$, we also have
\[
\HF^1(\Gal_{k(v)},M^{I_v})=0
\]
and hence,
\[
R(M/K_\cyc)\isomorph\HF^1_{\et}(\Spec \mathcal{O}_{\cyc},\eta_*M).
\]

Assume that $M$ is countable as a set. Recall from \cite[Prop. 2.1]{CoatesSujatha:FineSelmer} that
\[
\begin{aligned}
\HF^0(\Gal_S(K),(\pdual{M})_{K_\cyc}(1))&=0,\\
\pdual{\left(\bigoplus_{w\in S}\HF^0(\Gal_{K_w},(\pdual{M})_{K_\cyc}(1))\right)}&=\bigoplus_{v\in S_{K_\cyc}}\HF^2(\Gal_{(K_\cyc)_v},M)=0,
\end{aligned}
\]
where $\pdual{M}_{K_\cyc}$ is defined as in \eqref{eqn:def of coinduction}. By passing to the direct limit over all finite subsets of $M$ and all finite subextensions of $K_\cyc/K$, we obtain from the Poitou-Tate sequence an exact sequence of discrete $\Int_p$-modules
\begin{equation}\label{eqn:Poitou-Tate-I}
\begin{split}
0&\mto\HF^0(\Gal_S(K_\cyc),M)\mto\bigoplus_{v\in S_{K_\cyc}}\HF^0(\Gal_{(K_\cyc)_v},M)\\
\mto\pdual{\HF^2(\Gal_S(K),(\pdual{M})_{K_\cyc}(1))}&\mto\HF^1(\Gal_S(K_\cyc),M)\mto\bigoplus_{v\in S_{K_\cyc}}\HF^1(\Gal_{(K_\cyc)_v},M)\\
\mto\pdual{\HF^1(\Gal_S(K),(\pdual{M})_{K_\cyc}(1))}&\mto\HF^2(\Gal_S(K_\cyc),M)\mto 0
\end{split}
\end{equation}
Again, we use that $S_{K_\cyc}$ is a finite set.

Taking the dual, we also obtain an exact sequence of compact $\Int_p$-modules
\begin{equation}\label{eqn:Poitou-Tate-II}
\begin{split}
0\mto\pdual{\HF^2(\Gal_S(K_\cyc),M)}&\mto\HF^1(\Gal_S(K),(\pdual{M})_{K_\cyc}(1))\mto\bigoplus_{v\in S}\HF^1(\Gal_{K_v},(\pdual{M})_{K_\cyc}(1))\\
\mto\pdual{\HF^1(\Gal_S(K_\cyc),M)}&\mto\HF^2(\Gal_S(K),(\pdual{M})_{K_\cyc}(1))\mto\bigoplus_{v\in S}\HF^2(\Gal_{K_v},(\pdual{M})_{K_\cyc}(1))\\
\mto\pdual{\HF^0(\Gal_S(K_\cyc),M)}&\mto 0.
\end{split}
\end{equation}
%In particular, we see that
%\[
%Y_S(M/K_\cyc)=\ker\left(\HF^2(\Gal_S(K),(\pdual{M})_{K_\cyc}(1))\mto\bigoplus_{w\in S}\HF^2(\Gal_{K_W},(\pdual{M})_{K_\cyc}(1))\right).
%\]

We will derive yet another useful exact sequence. Recall that $Q$ denotes the quotient field of $\Omega$.

\begin{lem}\label{lem:coefficient sequence}
Assume that $S$ is a finite set of  primes of $K$ containing the primes above $p$. For any $\Gal_S(K)$-representation $M$ on a finite-dimensional vector space over $\FF_p$, there is an exact sequence
\[
\begin{aligned}
  0&\mto \HF^0(\Gal_S(K_\cyc),M)\\
  \mto\HF^1(\Gal_S(K),M_{K_\cyc})\mto Q\tensor_{\Omega}\HF^1(\Gal_S(K),M_{K_\cyc})&\mto\HF^1(\Gal_S(K_\cyc),M)\\ \mto\HF^2(\Gal_S(K),M_{K_\cyc})\mto Q\tensor_{\Omega} \HF^2(\Gal_S(K),M_{K_\cyc})&\mto \HF^2(\Gal_S(K_\cyc),M)\mto 0
  \end{aligned}
\]
\end{lem}
\begin{proof}
Choose a topological generator $\gamma$ of $\Gamma$ and set $t=\gamma-1$. Write $K_n$ for the intermediate field of $K_\cyc/K$ of degree $p^n$ over $K$. We then have
\[
Q/\Omega=\varinjlim_{n\in\N}\frac{1}{t^{p^n}}\Omega/\Omega
\]
Moreover, we have an isomorphism
\[
\begin{aligned}
\frac{1}{t^{p^n}}\Omega/\Omega&\xrightarrow{\isomorph}\Omega/t^{p^n}\Omega=\FF_p[\Gal(K_n/K)],\\
 \frac{a}{t^{p^n}}+\Omega&\mapsto a+t^{p^n}\Omega.
\end{aligned}
\]
For any compact $\Gal_S(K)$-module $N$, write $\cmplx{C}(\Gal_S(K),N)$ for the continuous homogenous cochain complex.
Clearly, $M_{K_\cyc}$ is finitely generated and free as $\Omega$-module. Hence, there is an isomorphism
\[
\frac{1}{t^{p^n}}\Omega/\Omega\Ltensor_{\Omega}\cmplx{C}(\Gal_S(K),M_{K_\cyc})\isomorph \cmplx{C}(\Gal_S(K_n),M)
\]
in the derived category of complexes of $\Omega$-modules \cite[Prop. 1.6.5]{FK:CNCIT} and one checks that the map induced from the inclusion
\[
\frac{1}{t^{p^n}}\Omega/\Omega\subset \frac{1}{t^{p^{n+1}}}\Omega/\Omega
\]
agrees with the restriction map
\[
\cmplx{C}(\Gal_S(K_n),M)\mto\cmplx{C}(\Gal_S(K_{n+1}),M).
\]
Therefore, we obtain an isomorphism
\[
Q/\Omega\Ltensor_{\Omega}\cmplx{C}(\Gal_S(K),M_{K_\cyc})\isomorph \cmplx{C}(\Gal_S(K_\cyc),M)
\]
in the derived category of complexes of $\Omega$-modules. We take the derived tensor product over $\Omega$ of the exact sequence
\[
0\mto\Omega\mto Q\mto Q/\Omega\mto 0
\]
with $\cmplx{C}(\Gal_S(K),M_{K_\cyc})$ and consider the associated long exact cohomology sequence of the resulting distinguished triangle. Since $Q$ is a flat $\Omega$-module, we conclude
\[
Q\tensor_{\Omega}\HF^n(\Gal_S(K),M_{K_\cyc})=\HF^n(Q\Ltensor_{\Omega}\cmplx{C}(\Gal_S(K),M_{K_\cyc})).
\]
Finally, we recall that
\[
\HF^0(\Gal_S(K),M_{K_\cyc})=0.
\]
\end{proof}

\section{Conjecture A and Equivalent Formulations}\label{sec:Conjecture A}

Let $E$ be an elliptic curve over the number field $K$ and $S$ a finite set of primes of $K$ containing the primes dividing $p$ and the primes of bad reduction of $E$. Recall \cite[Conjecture A]{CoatesSujatha:FineSelmer} for the group $\HF^2(\Gal_S(K),(\Tate_p E)_{K_\cyc})$:

\begin{conj}[$A(K,E{[p]})$]
The group $\HF^2(\Gal_S(K),(\Tate_p E)_{K_\cyc})$ is finitely generated as a $\Int_p$-module. Equivalently, the Pontryagin dual $Y(E[p^\infty]/K_\cyc)$ of the fine Selmer group $R(E[p^\infty]/K_\cyc)$ is finitely generated as a $\Int_p$-module.
\end{conj}

In fact, as Greenberg points out in \cite[Prop. 4.1.6]{Greenberg:ProjectiveModules}, Conjecture $A(K,E[p])$ is also equivalent to the vanishing of $\HF^2(\Gal_S(K_\cyc),E[p])$. In this form, it admits the following sensible generalisation for any $\Gal_K$-representation $M$ on a finite-dimensional vector space over $\FF_p$:

\begin{conj}[$A(K,M)$]
Let $S$ be any finite set of  primes of $K$ containing the primes above $p$ and the primes where $M$ is ramified. Then
\[
 \HF^2(\Gal_S(K_\cyc),M)=0.
\]
\end{conj}

We compile a list of equivalent reformulations.

\begin{thm}\label{thm:equivalent forms of Conjecture A}
Let $K$ be a number field and $M$ be a $\Gal_K$-representation on a finite-dimensional vector space over $\FF_p$. Assume that
$S$ is a finite set of  primes of $K$ containing the primes above $p$ and the primes where $M$ is ramified. The following statements are equivalent:
\begin{enumerate}
 \item\label{enum:A1} $\HF^2(\Gal_S(K_\cyc),M)=0$.
 \item\label{enum:A2} $\HF^2(\Gal_S(K),M_{K_\cyc})$ is finite.
 \item\label{enum:A3} $\rk_{\Omega}\HF^1(\Gal_S(K),M_{K_\cyc})=r_{M}$.
 \item\label{enum:A4} $\rk_{\Omega}\pdual{\HF^1(\Gal_S(K_\cyc),M)}=r_{M}$.
 \item\label{enum:A5} $\HF^1(\Gal_S(K),\pdual{M}_{K_\cyc}(1))\mto\bigoplus_{v\in S}\HF^1(\Gal_{K_v},\pdual{M}_{K_\cyc}(1))$ is injective.
 \item\label{enum:A6} The fine Selmer group $R(\pdual{M}(1)/K_\cyc)$ is finite.
 \item\label{enum:A7} The inflation map $\HF^2(\Gal_S(K_\cyc),M)\mto\HF^2(\Gal_T(K_\cyc),M)$ is injective for every finite set of  primes $T$ containing $S$.
 \item\label{enum:A8} The Gysin map $\HF^1(\Gal_{K_\cyc},M)\mto\bigoplus_{v\not\in S_{K_\cyc}}\HF^0(\Gal_{k(v)},M(-1))$ is surjective.
 \end{enumerate}
 \end{thm}
 \begin{proof}
 The equivalence of \eqref{enum:A1}, \eqref{enum:A2}, \eqref{enum:A3}, \eqref{enum:A4} is an immediate consequence of the sequence in Lemma~\ref{lem:coefficient sequence} and the Euler characteristic formula
\[
 \rk_{\Omega}\HF^1(\Gal_S(K),M_{K_\cyc})-\rk_{\Omega}\HF^2(\Gal_S(K),M_{K_\cyc})=r_M.
\]
The equivalence of \eqref{enum:A1} and \eqref{enum:A5} follows easily from the Poitou-Tate sequence \eqref{eqn:Poitou-Tate-II}.
Similarly, we may consider the Poitou-Tate sequence \eqref{eqn:Poitou-Tate-I}, with $M$ replaced by $\pdual{M}(1)$. Since the first two terms are finite groups, the equivalence of \eqref{enum:A2} and \eqref{enum:A6} is immediate.

The equivalence of \eqref{enum:A1}, \eqref{enum:A7}, and \eqref{enum:A8} can be seen as follows. Recall that
\[
\HF^k(\Gal_{K_\cyc},M)=\varinjlim_{T}\HF^k(\Gal_T(K_\cyc),M)
\]
where $T$ runs through all finite sets of primes of $K$ containing $S$. Moreover,
\[
\HF^2(\Gal_{K_\cyc},M)=0
\]
by \cite[Cor. 8.1.18]{NSW:CohomNumFields}. Trivially, \eqref{enum:A1} implies \eqref{enum:A7}. If we assume \eqref{enum:A7}, then $\HF^2(G_S(K_\cyc),M)$ injects into $\HF^2(\Gal_{K_\cyc},M)$ and is therefore the trivial group. So, \eqref{enum:A1} and \eqref{enum:A7} are equivalent.

Passing to the direct limit over $T$ of the Gysin sequences \eqref{eqn:Gysin sequence} we obtain the exact sequence
\[
\begin{split}
0&\mto\HF^1(\Gal_S(K_\cyc),M)\mto\HF^1(\Gal_{K_\cyc},M)\mto\bigoplus_{v\not\in S_{K_\cyc}}\HF^0(\Gal_{k(v)},M(-1))\\
 &\mto\HF^2(\Gal_S(K_\cyc),M)\mto 0
\end{split}
\]
From this sequence, we can easily read off the equivalence of \eqref{enum:A1} and \eqref{enum:A8}.
\end{proof}

\begin{rem}
As $R(\pdual{M}(1)/K_\cyc)$ does not depend on $S$, we also see that Conjecture $A(K,M)$ is independent of the particular choice of $S$.
\end{rem}

Recall that $M^{\sesi}$ denotes the semisimplification of $M$.

\begin{lem}\label{lem:A under quotiens and extensions}
Let $M$ be a representation of $\Gal_K$ on a finite-dimensional vector space over $\FF_p$ and $N\subset M$ be a subrepresentation.
\begin{enumerate}
\item $A(K,M)$ implies $A(K,M/N)$.
\item The conjunction of $A(K,N)$ and $A(K,M/N)$ imply $A(K,M)$.
\item $A(K,M^{\sesi})$ implies $A(K,M)$.
\end{enumerate}
\end{lem}
\begin{proof}
The first two assertions follow easily from the long exact cohomology sequence associated to the short exact sequence
\[
0\mto N\mto M\mto M/N\mto 0
\]
and the fact that the $p$-cohomological dimension of $\Gal_S(K_\cyc)$ is at most $2$. The third assertion follows from the second by induction over the filtration steps of a Jordan-H\"older series for $M$.
\end{proof}

\begin{cor}
Assume that $E/K$ has a $p$-torsion point defined over $K$. Then the conjunction of $A(K,\mu_p)$ and $A(K,\FF_p)$ implies $A(K,E[p])$ and $A(K,E[p])$ implies $A(K,\mu_p)$. If in addition $K$ contains $\mu_p$, then $A(K,\mu_p)$, $A(K,\FF_p)$, and $A(K,E[p])$ are equivalent.
\end{cor}
\begin{proof}
The $p$-torsion point generates a submodule $T$ of $E[p]$ of dimension $1$ over $\FF_p$ with trivial $\Gal_K$-action. Moreover, the Weil pairing on $E[p]$ implies that $E[p]/T\isomorph \mu_p$. Hence, we have an exact sequence of $\Gal_K$-modules
\[
0\mto \FF_p\mto E[p]\mto \mu_p\mto 0.
\]
The claim now follows from Lemma~\ref{lem:A under quotiens and extensions}.
\end{proof}

\begin{rem}
We will explain in Proposition~\ref{prop:relation A and mu equal zero} the relation of $A(K,\mu_p)$ and $A(K,\FF_p)$ with Iwasawa's conjecture on the vanishing of the $\mu$-invariant.
\end{rem}

We investigate how Conjecture $A(K,M)$ behaves under field extensions.

\begin{lem}\label{lem:A under field extensions}
Let $L/K$ be a finite field extension and $M$ a $\Gal_K$-representation on a finite-dimensional vector space over $\FF_p$.
\begin{enumerate}
 \item\label{enum:field extension not Galois} If Conjecture $A(L,M)$ is true, then so is $A(K,M)$.
 \item\label{enum:field extension induced} If $M$ is the induced representation of a $\Gal_L$-representation $N$, then Conjecture $A(K,M)$ is true if and only if $A(L,N)$ is true.
 \item\label{enum:field extension Galois} If $L/K$ is Galois with Galois group $\Delta$, then $A(L,M)$ is true if and only if $A(K,\rho\tensor_{\FF_p}M)$ is true for every finite-dimensional (every simple) $\FF_p$-representation $\rho$ of $\Delta$.
 \item\label{enum:field extension pro p} If $L/K$ is a Galois extension of order a power of $p$, then $A(L,M)$ is true if and only if $A(K,M)$ is true.
\end{enumerate}
 \end{lem}
\begin{proof}
Choose $S$ large enough such that $L/K$ and $M$ are unramified outside $S$. Since the $p$-cohomological dimension of $\Gal_S(K_\cyc)$ is less or equal to $2$, the corestriction map
\[
\HF^2(\Gal_S(L_\cyc),M)\mto\HF^2(\Gal_S(K_\cyc),M)
\]
is surjective \cite[Prop. 3.3.11]{NSW:CohomNumFields}. This proves \eqref{enum:field extension not Galois}.

Let $n$ be the degree of $L\cap K_\cyc/K$ and assume that $M$ is induced by the $\Gal_L$-representation $N$. Let $M'$ be the $\Gal_{K_\cyc}$-representation induced by the $\Gal_{L_\cyc}$-representation $N$. Then $M=(M')^n$ as $\Gal_{K_\cyc}$-representations. By Shapiro's Lemma we have
\[
\HF^2(\Gal_S(L_\cyc),N^n)=\HF^2(\Gal_S(K_\cyc),(M')^n)
\]
Hence $A(K,M)$ is equivalent to $A(L,N^n)$, which is clearly equivalent to $A(L,N)$. This proves
\eqref{enum:field extension induced}.

We prove \eqref{enum:field extension Galois}. Let $L/K$ be Galois with Galois group $\Delta$ and $\rho$ a $\FF_p$-representation of $\Delta$ of dimension $k$ over $\FF_p$. Then $\rho$ is trivial as a $\Gal_S(L_\cyc)$-representation. In particular,
\[
\HF^2(\Gal_S(L_\cyc),\rho\tensor_{\FF_p}M)=\HF^2(\Gal_S(L_\cyc),M)^{k}
\]
and Conjecture $A(L,M)$ is true if and only if Conjecture $A(L,\rho\tensor_{\FF_p}M)$ is true. Combining this with \eqref{enum:field extension not Galois}, we see that $A(L,M)$ implies $A(K,M\tensor_{\FF_p}\rho)$.

Suppose that $A(K,M\tensor_{\FF_p}\rho)$ is true for every simple $\FF_p$-representation of $\Delta$. The induced $\Gal_K$-representation of the restriction of $M$ to $\Gal_L$ is $\FF_p[\Delta]^\sharp\tensor_{\FF_p}M$. By induction on the length of a Jordan-H\"older series of $\FF_p[\Delta]$ we see that $A(K,\FF_p[\Delta]^\sharp\tensor_{\FF_p}M)$ is true. Hence, $A(L,M)$ is also true by \eqref{enum:field extension induced}.

For the proof of \eqref{enum:field extension pro p} it remains to note that $\FF_p$ is the only simple $\Delta$-representation if the order of $\Delta$ is a power of $p$.
\end{proof}

The following proposition is essentially well-known, see for example \cite[Lem.~8]{Schneider:HeightPairingII}.

\begin{prop}
The following are equivalent:
\begin{enumerate}
 \item\label{enum:A all fields all M} Conjecture $A(K,M)$ holds for all number fields $K$ and all $\Gal_K$-repre\-sen\-ta\-tions $M$.
 \item\label{enum:A all simple M} Conjecture $A(K,M)$ holds for a fixed field $K$ and all simple $\Gal_K$-repre\-sen\-ta\-tions $M$.
 \item\label{enum:A all fields} Conjecture $A(L,M)$ holds for all finite extensions $L/K$ for a fixed field $K$ and a fixed $\Gal_K$-repre\-sen\-ta\-tion $M$.
 \item\label{enum:A p-dim one exist} The $p$-cohomological dimension of $\Gal_S(K_\cyc)$ is $1$ for some number field $K$ and some finite set $S$ of primes containing the primes above $p$.
 \item\label{enum:A p-dim one for all} The $p$-cohomological dimension of $\Gal_S(K_\cyc)$ is $1$ for every number field $K$ and every finite set $S$ of primes containing the primes above $p$.
\end{enumerate}
\end{prop}
\begin{proof}
The implications
\[
\eqref{enum:A p-dim one for all}\Rightarrow\eqref{enum:A p-dim one exist}\Rightarrow\eqref{enum:A all fields all M}\Rightarrow\eqref{enum:A all simple M}
\]
are trivial.

Assume \eqref{enum:A all simple M}. Then Conjecture $A(K,M)$ holds for all $\Gal_K$-representations $M$ by Lemma~\ref{lem:A under quotiens and extensions}.(3). Let $L/K$ be a finite extension and $M$ a fixed $\Gal_K$-representation. Let $N$ be the $\Gal_K$-representation which is induced from the restriction of $M$ to $\Gal_L$, so that $A(K,N)$. Assertion \eqref{enum:A all fields} now follows from Lemma~\ref{lem:A under field extensions}.(2).

Assume \eqref{enum:A all fields} for a fixed number field $K'$ and a fixed $\Gal_{K'}$-representation $M$ and let $L$ be a finite extension of $K'$ such that the restriction of $M$ to $L$ is trivial. Then $A(L,M)$ implies $A(L,\FF_p)$ by Lemma~\ref{lem:A under quotiens and extensions}.(3). By Lemma~\ref{lem:A under field extensions}.(1), this implies $A(L',\FF_p)$ for all subfields $L'$ of $L$. In particular, we have $A(K',\FF_p)$ for every number field $K'$. Fix a number field $K$ and let $\mathcal{L}\subset K_S$ be the field fixed by a $p$-Sylow subgroup of $\Gal_S(K_\cyc)$. Then
\[
\HF^2(\Gal_S(\mathcal{L}),\FF_p)=\varinjlim_{L/K}\HF^2(\Gal_S(L_\cyc),\FF_p)=0,
\]
with $L/K$ running through the finite subextensions of $\mathcal{L}/K$. Since $\Gal_S(\mathcal{L})$ is a pro-$p$ group of  $p$-cohomological dimension at most $2$, we conclude
\[
\cd_p\Gal_S(\mathcal{L})\leq 1.
\]
Note that the $p$-cohomological dimensions of $\Gal_S(K_\cyc)$ and any of its $p$-Sylow subgroups are the same. On the other hand, we cannot have $\cd_p \Gal_S(K_\cyc)=0$, as we know that  $\cd_p \Gal_S(K)=2$ and that
\[
\cd_p \Gal_S(K)\leq \cd_p \Gal_S(K_\cyc) + \cd_p \Gamma=\cd_p \Gal_S(K_\cyc)+1
\]
 \cite[Prop. 10.11.3, Prop. 3.3.8]{NSW:CohomNumFields}. We have thus proved the remaining implication
\(
\eqref{enum:A all fields}\Rightarrow\eqref{enum:A p-dim one for all}.
\)
\end{proof}

Conjecture A has been generalised in various other directions, for example by allowing more general coefficient rings \cite{Lim:Notes} or by considering Hida deformations and `admissible' $p$-adic Lie extensions \cite{JhaSujatha:HidaDefFineSelmer}, \cite{Jha:FineSelmerHidaDefNonCom}. In the end, however, these generalisations turn out to be equivalent to Conjecture $A(K,M)$ for suitable $M$, as the following result shows.

\begin{prop}\label{prop:residual reps are enough}
Let $R$ be a possibly non-commutative adic $\Int_p$-algebra, i.\,e.\ compact for the topology defined by the powers of the Jacobson radical $\Jac_R\subset R$, and let $M$ be a finitely generated, compact left $R$-module with a continuous, $R$-linear action of $\Gal_K$ which is unramified outside a finite set $S$ of  primes of $K$ containing the primes above $p$. Assume that $\mathcal{L}/K$ is a Galois extension inside $K_S/K$ with Galois group $\mathcal{G}=\Gal(\mathcal{L}/K)$ such that
\begin{enumerate}
\item $K_\cyc\subset \mathcal{L}$,
\item $\mathcal{H}=\Gal(\mathcal{L}/K_\cyc)$ is a topologically finitely generated pro-$p$-group.
\end{enumerate}
Then $\HF^2(\Gal_S(K),M_{\mathcal{L}})$ is finitely generated over the profinite group ring $R[[\mathcal{H}]]$ if and only if $A(K,M/\Jac_R M)$ holds. If in addition, $R[[\mathcal{H}]]$ is Noetherian, then $Y(\pdual{M}(1)/\mathcal{L})$ is finitely generated over $R[[\mathcal{H}]]$ if and only if $A(K,M/\Jac_R M)$ holds.
\end{prop}
\begin{proof}
Under our assumptions on $\mathcal{H}$, both $\mathcal{G}$ and $\mathcal{H}$ are topologically finitely generated pro-$p$ groups. Hence, both $R[[\mathcal{H}]]$ and $R[[\mathcal{G}]]$ are adic $\Int_p$-algebras \cite[Prop 3.2]{Witte:MCVarFF}. The Jacobson radical of $R[[\mathcal{H}]]$ is given by
\[
\Jac_{R[[\mathcal{H}]]}=\ker\left(R[[\mathcal{H}]]\mto R/\Jac_R\right).
\]
Since
\[
\Jac_{R[[\mathcal{H}]]}/\Jac_{R[[\mathcal{H}]]}^2\subset R[[\mathcal{H}]]/\Jac_{R[[\mathcal{H}]]}^2
\]
is finite, the Jacobson radical is finitely generated as a left or right $R[[\mathcal{H}]]$-module by the topological Nakayama lemma. Hence,
\[
\Jac_{R[[\mathcal{H}]]}R[[\mathcal{G}]]=\ker\left(R[[\mathcal{G}]]\mto R/\Jac_R[[\Gamma]]\right)
\]
is also finitely generated as a left or right $R[[\mathcal{G}]]$-module. In particular,
\[
R/\Jac_R[[\Gamma]]=R/\Jac_R\tensor_{R[[\mathcal{H}]]}R[[\mathcal{G}]]
\]
is finitely presented as a right $R[[\mathcal{G}]]$-module. Since $\HF^2(\Gal_S(K),-)$ is a right exact functor that commutes with finite products, we conclude from the choice of a finite free presentation that
\[
\begin{aligned}
R/\Jac_R\tensor_{R[[\mathcal{H}]]}\HF^2(\Gal_S(K),M_{\mathcal{L}})&=
R/\Jac_R[[\Gamma]]\tensor_{R[[\mathcal{G}]]}\HF^2(\Gal_S(K),M_{\mathcal{L}})\\
&=\HF^2(\Gal_S(K),R/\Jac_R[[\Gamma]]\tensor_{R[[\mathcal{G}]]}\tensor_{R[[\mathcal{G}]]}M_{\mathcal{L}})\\
&=\HF^2(\Gal_S(K),(M/\Jac_R M)_{\mathcal{K_\cyc}})
\end{aligned}
\]
The topological Nakayama lemma then implies that the compact $R[[\mathcal{H}]]$-module $\HF^2(\Gal_S(K),M_{\mathcal{L}})$ is finitely generated precisely if $\HF^2(\Gal_S(K),(M/\Jac_R M)_{\mathcal{K_\cyc}})$ is finite. The latter is equivalent to Conjecture $A(K,M/\Jac_R M)$ by Theorem~\ref{thm:equivalent forms of Conjecture A}.

Now assume that $R[[\mathcal{H}]]$ is Noetherian and note that
\[
Y(\pdual{M}(1)/\mathcal{L})=\ker\left(\HF^2(\Gal_S(K),M_{\mathcal{L}})\mto\bigoplus_{v\in S}\HF^2(\Gal_{K_v},M_{\mathcal{L}})\right).
\]
In particular, if $\HF^2(\Gal_S(K),M_{\mathcal{L}})$ is finitely generated over $R[[\mathcal{H}]]$, then so is its submodule $Y(\pdual{M}(1)/\mathcal{L})$. Conversely, assume $Y(\pdual{M}(1)/\mathcal{L})$ is finitely generated over $R[[\mathcal{H}]]$. In order to imply that $\HF^2(\Gal_S(K),M_{\mathcal{L}})$ is finitely generated over $R[[\mathcal{H}]]$, it is then sufficient to prove that $\HF^2(\Gal_{K_v},M_{\mathcal{L}})$ is finitely generated over $R[[\mathcal{H}]]$ for each  prime $v$ of $K$.

By the same argument as before, this will follow if we show that the group $\HF^2(\Gal_{K_v},(M/\Jac_R M)_{K_\cyc})$ is finite. But this is true, since by local duality,
\[
\pdual{\HF^2(\Gal_{K_v},(M/\Jac_R M)_{K_\cyc})}=\bigoplus_{w\mid v}\HF^0(\Gal_{(K_{\cyc})_w},\pdual{M/\Jac_R M}(1))
\]
and since in $K_\cyc/K$, every  prime splits into finitely many primes.
\end{proof}

\begin{rem}\
\begin{enumerate}
\item The ring $R[[\mathcal{H}]]$ is Noetherian if $R$ is a commutative adic $\Int_p$-algebra and $\mathcal{H}$ is a compact $p$-adic Lie group \cite[Cor. 3.4]{Witte:Splitting}.
\item In general, $Y(\pdual{M}(1)/\mathcal{L})$ is expected to be rather small. If $\mathcal{G}$ is a compact $p$-adic Lie group of dimension greater than $1$, then it is conjectured that $Y(E[p^\infty]/\mathcal{L})$ is in fact a finitely generated torsion module over $\Int_p[[\mathcal{H}]]$ \cite[Conjecture B]{CoatesSujatha:FineSelmer}.
\item One may also conjecture that $\HF^2(\Gal_S(\mathcal{L}),M)=0$ for a $\Gal_K$-representation $M$ on a finite-dimensional vector space over $\FF_p$ and with $\mathcal{L}$ as in the above proposition. This is clearly implied by Conjecture $A(K,M)$, but it seems to be considerably weaker than $A(K,M)$ if $\mathcal{L}/K_\cyc$ is of infinite degree.
\end{enumerate}
\end{rem}

Conjecture $A(K,M)$ is closely related to Iwasawa's classical conjecture on the vanishing of the $\mu$-invariant for the cyclotomic extension of any number field $K$. The precise relationship is as follows:  We consider the Iwasawa modules
\begin{align*}
 X_{\nr}(K)&=\varprojlim_{K\subset L\subset K_\cyc} \Cl(L)\tensor_{\Int}\Int_p,\\
 X_{\cs}(K)&=\varprojlim_{K\subset L\subset K_\cyc} \Cl_p(L)\tensor_{\Int}\Int_p,\\
 X_S(K)&=\Gal(K_S(p)/K)^{\ab},
\end{align*}
where $\Cl_p(L)$ is the Picard group of the ring of integers $\IntR_L[\frac{1}{p}]$ of $L$ with $p$ inverted and $K_S(p)$ is the maximal pro-$p$-Galois extension of $K$ unramified outside the finite set $S$ of  primes containing the primes above $p$. It is known that $X_{\nr}(K)$, $X_{\cs}(K)$ and $X_S(K)$ are finitely generated $\Lambda$-modules. While $X_{\nr}(K)$ and $X_{\cs}(K)$ are $\Lambda$-torsion, the $\Lambda$-rank of $X_S(K)$ is given by the number $r_2$ of complex places of $K$ \cite[Prop. 11.1.4, Prop. 11.3.1, Cor. 11.3.15]{NSW:CohomNumFields}. For any finitely generated $\Lambda$-module $M$, we write $\mu(M)$ for the $\mu$-invariant of $M$. We then have
\[
 \mu(X_{\nr}(K))=\mu(X_{\cs}(K))
\]
and if the $p$-th roots of unity $\mu_p$ are contained in $K$, we also have
\[
 \mu(X_S(K))=\mu(X_{\nr}(K))
\]
\cite[Cor. 11.3.16, 11.3.17]{NSW:CohomNumFields}. Iwasawa's classical conjecture on the $\mu$-invariant amounts to
\[
\mu(X_{\nr}(K))=0
\]
for all number fields $K$.

\begin{prop}\label{prop:relation A and mu equal zero}
Let $K$ be a number field.
\begin{enumerate}
 \item\label{enum:Xnr} Conjecture $A(K,\mu_p)$ holds if and only if $\mu(X_{\nr}(K))=0$.
 \item\label{enum:XS} Conjecture $A(K,\FF_p)$ holds if and only if $\mu(X_S(K))=0$ for some finite $S$ containing the primes above $p$.
\end{enumerate}
In particular, Iwasawa's conjecture on the $\mu$-invariant holds for all number fields $K$ precisely if Conjecture $A(K,M)$ holds for all $K$ and $M$.
\end{prop}
\begin{proof}
We prove \eqref{enum:Xnr}. Let $\Sigma$ be the set of primes of $K$ above $p$. Then class field theory in combination with local duality and the Poitou-Tate sequence implies
\[
\begin{aligned}
\Cl_p(K)\tensor_{\Int}\Int/(p^n)&=\coker\left(\bigoplus_{v\in \Sigma}\Gal_{K_v}^{\ab}\tensor_{\Int}\Int/(p^n)\mto \Gal_\Sigma(K)^{\ab}\tensor_{\Int}\Int/(p^n)\right)\\
&=\coker\left(\bigoplus_{v\in \Sigma}\HF^1(\Gal_{K_v},\mu_{p^n})\mto \pdual{\HF^1(\Gal_S(K),\Int/(p^n))}\right)\\
&=\ker\left(\HF^2(\Gal_\Sigma(K),\mu_{p^n})\mto \bigoplus_{v\in\Sigma} \HF^2(\Gal_{K_v},\mu_{p^n})\right).
\end{aligned}
\]
Passing to the inverse limit over $n\in N$ and all finite subextensions of $K_\cyc/K$ we conclude
\[
X_{\cs}(K)=Y((\Rat_p/\Int_p)/K_\cyc).
\]
Since $X_{\cs}(K)$ is $\Lambda$-torsion, we have $\mu(X_{\nr}(K))=\mu(X_{\cs}(K))=0$ precisely if $X_{\cs}(K)$ is finitely generated over $\Int_p$. The equivalence of $A(K,\mu_p)$ with $\mu(X_{\nr}(K))=0$ then follows from Proposition~\ref{prop:residual reps are enough} with $R=\Int_p$, $\mathcal{L}=K_\cyc$ and $M=\Int_p(1)$.

For \eqref{enum:XS}, we use that $\mu(X_S(K))=0$ is equivalent to the maximal pro-$p$-quotient $\Gal_S(K_\cyc)(p)$ of $\Gal_S(K_\cyc)$ being a free pro-$p$-group, which is in turn equivalent to
\[
\HF^2(\Gal_S(K_\cyc)(p),\FF_p)=0.
\]
Moreover,
\[
\HF^2(\Gal_S(K_\cyc)(p),\FF_p)=\HF^2(\Gal_S(K_\cyc),\FF_p)
\]
\cite[Thm. 11.3.7, Prop. 3.9.5, Cor. 10.4.8]{NSW:CohomNumFields}.
\end{proof}

\begin{rem}
Assume that $T$ is a representation of $\Gal_K$ on a finitely generated free $\Int_p$-module which is unramified outside a finite set of primes $S$. The weak Leopoldt conjecture for $T$ over $K_\cyc$ states that
\[
\HF^2(\Gal_S(K_\cyc),\pdual{T}(1))=0,
\]
see \cite{Greenberg:StructureOfSelmerGroups} and the references given therein. We note that this conjecture is implied by $A(K,\pdual{(T/pT)}(1))$. Indeed, assuming $A(K,\pdual{(T/pT)}(1))$, the long exact cohomology sequence for the short exact sequence
\[
0\mto\pdual{(T/pT)}(1)\mto\pdual{T}(1)\xrightarrow{\cdot p}\pdual{T}(1)\mto 0
\]
shows that the group $\HF^2(\Gal_S(K_\cyc),\pdual{T}(1))$ is uniquely $p$-divisible and $p$-torsion, and hence trivial.
\end{rem}

To view Conjecture $A(K,M)$ in this framework, see also \cite{Sujatha:EllipticCurvesAndMuEqualZero}.

\section{The Isogeny Invariance Conjecture}\label{sec:IsogenyInvariance}

In \cite{CoatesSujatha:FineSelmer}, the authors express their belief that Conjecture $A(K,E[p])$ is isogeny invariant in the following sense:

\begin{conj}[$I(K,E)$]
Let $E$ be a fixed elliptic curve over $K$. For every nontrivial isogeny $f\colon E\mto E'$ to an elliptic curve $E'$, with both $f$ and $E'$ defined over $K$, $A(K,E[p])$ holds if and only if $A(K,E'[p])$ holds.
\end{conj}

In this section, we will discuss this conjecture. Let us first recall the following fact on isogenies.

\begin{prop}\label{prop:classification}
Let $E/K$ and  $E'/K$ be elliptic curves which are isogenous over $K$. Then precisely one of the following holds true:
\begin{enumerate}
\item[(a)] $E$ has no $\Gal_K$-submodule of order $p$, $E[p]$ is a simple $\Gal_K$-module and there exists a $\Gal_K$-isomorphism $E'[p]\isomorph E[p]$.
\item[(b)]  $E$ has a $\Gal_K$-submodule $T\isomorph \pdual{T}(1)$ of order $p$ and both $E[p]$ and $E'[p]$ are extensions of $T$ by $T$.
\item[(c)] $E$ has complex multiplication over $K$ by an imaginary quadratic extension of $\Rat$ in which $p$ splits completely and $\Tate_p E\isomorph\chi\oplus \chi^{-1}(1)\isomorph \Tate_p E'$ for some character $\chi\colon \Gal_K\mto \Int_p^{\times}$ such that the residual representation of $\chi^{-2}(1)$ is nontrivial.
\item[(d)] $E$ has a $\Gal_K$-submodule $T\not\isomorph \pdual{T}(1)$ of order $p$ and there exists an elliptic curve $E_0/K$ isogenous to $E$ over $K$ with a unique finite maximal cyclic $\Gal_K$-submodule $C\subset E_0$ of $p$-power order. Set $E_1=E_0/C$. Then there exist non-split exact sequences
    \[
    \begin{array}{lcccccccr}
  0&\mto & T          &\mto& E_0[p]&\mto& \pdual{T}(1)&\mto& 0,\\
  0&\mto &\pdual{T}(1)&\mto& E_1[p]&\mto& T           &\mto& 0.
  \end{array}
  \]
  Moreover, $E'[p]$ is either isomorphic to $E_0[p]$, to $E_1[p]\not\isomorph E_0[p]$, or to $T\oplus \pdual{T}(1)$. The last possibility can occur if and only if the order of $C$ is greater than $p$.
\end{enumerate}
\end{prop}
\begin{proof}
Since $E$ and $E'$ are isogenous over $K$, the group of continuous $\Gal_K$-homo\-morphisms $\Hom_{\Gal_K}(\Tate_p E,\Tate_p E')$ is nontrivial \cite[Thm. 7.4]{Silverman:EllipticCurves}. Moreover, the $\Int_p$-module $\Hom_{\Gal_K}(\Tate_p E,\Tate_p E')$ is a finitely generated and free. The long exact $\Ext_{\Gal_K}$-sequence for the short exact sequence
\[
0\mto \Tate_p E'\xrightarrow{\cdot p} \Tate_p E'\mto E'[p]\mto 0
\]
gives us a left exact sequence
\[
0\mto \Hom_{\Gal_K}(\Tate_p E,\Tate_p E')\xrightarrow{\cdot p}\Hom_{\Gal_K}(\Tate_p E,\Tate_p E')\mto\Hom_{\Gal_K}(\Tate_p E,E'[p]).
\]
The corresponding short exact sequence for $\Tate_p E$ gives us an isomorphism
\[
\Hom_{\Gal_K}(\Tate_p E,E'[p])\isomorph\Hom_{\Gal_K}(E[p],E'[p]).
\]
In particular, $\Hom_{\Gal_K}(E[p],E'[p])$ is nontrivial. Since $E[p]$ and $E'[p]$ are groups of the same order $p^2$, there exists either an $\Gal_K$-equivariant isomorphism $E[p]\isomorph E'[p]$ or a $\Gal_K$-equivariant homomorphism
\[
f\colon E[p]\mto E'[p]
\]
with kernel $\ker(f)$ of order $p$. In the latter case, the nondegenerate, alternating Weil pairings on $E[p]$ and $E'[p]$ imply the existence of a $\Gal_K$-equivariant homomorphism $\pdual{f}\colon E'[p]\mto E[p]$ with kernel isomorphic to $\pdual{\ker(f)}(1)$. Moreover, we have
\[
\begin{aligned}
\ker(\pdual{f})\isomorph\im(f)\isomorph\pdual{\ker(f)}(1),\\
\ker(f)\isomorph\im(\pdual{f})\isomorph\pdual{\ker(\pdual{f})}(1).
\end{aligned}
\]
In any case, we conclude that $E[p]^{\sesi}=E'[p]^{\sesi}$.

If $E$ has no finite $\Gal_K$-submodule of order $p$, $E[p]$ must be a simple $\Gal_K$ module. Hence, an isogeny $f$ as above cannot exist. In particular, $E[p]\isomorph E'[p]$.

Henceforth, we may assume that $E[p]$ has a finite submodule $T$ of order $p$. If $T\isomorph\pdual{T}(1)$, then $E[p]^\sesi=T^2$. In particular, both $E[p]$ and $E'[p]$ are extensions of $T$ by itself. So, we may assume that $T\not\isomorph\pdual{T}(1)$.

Assume that $\Tate_p E$ contains a $\Gal_K$-stable submodule $X$ of rank $1$ over $\Int_p$. Then $V=\Rat_p\tensor_{\Int}\Int_p$ splits into two one-dimensional representations and necessarily, the image of $\Gal_K$ in the automorphism group of $\Tate_p E$ is abelian. Then $E$ must have complex multiplication over $K$ by an order $\IntR$ in an imaginary quadratic number field $L$. Indeed, the $\Int_p$-module  $\End_{\Gal_K}(\Tate_p E)$ of $\Gal_K$-equivariant endomorphisms of $\Tate_p E$, i.\,e.\ those endomorphisms that commute with elements in the image of $\Gal_K$, is strictly greater than the centre of $\End_{\Int_p}(\Tate_p E)$, since it contains the centre and the image of $\Gal_K$. (If the image of $\Gal_K$ were contained in the centre, which is not possible, then $\End_{\Gal_K}(\Tate_p E)=\End_{\Int_p}(\Tate_p E)$.) By a celebrated theorem of Faltings \cite[Thm. 7.7]{Silverman:EllipticCurves}, this implies that the endomorphism ring of $E$ must be larger than $\Int$, as well. In particular, $V$ is a module of rank $1$ over $\Rat_p\tensor_{\Rat} L$ and $\Gal(\mathcal{L}/K)$ may be identified with an open subgroup of $(\Int_p\tensor_{\Int}\IntR)^\times$. But then, $X$ can only exist if $\Rat_p\tensor_{\Rat} L$ is not a field. So, $p$ must split in $L$. In this case, $\Tate_p E\isomorph \chi\oplus \chi^{-1}(1)$ for a $1$-dimensional representation $\chi$. The residual representation of $\chi^{-2}(1)$ cannot be trivial by our assumption that $T\not\isomorph\pdual{T}(1)$. Moreover, the same decomposition also holds for $\Tate_p E'$.

Henceforth, we may assume that $\Tate_p E$ contains no $\Gal_K$-stable submodule of rank $1$ over $\Int_p$. In particular, maximal cyclic $\Gal_K$-submodules of order a power of $p$ exist in $E$ and all $E'$ isogenous to $E$. If $E$ contains a further submodule $T'$ of order $p$ besides $T$, then necessarily $T'\isomorph \pdual{T}(1)$ and there exists a maximal cyclic submodule $C'\subset E$ of $p$-power order that contains $T'$. Consider the isogeny $\pi\colon E\mto E/C'=E_0$. The preimage of any submodule $T_0$ of $E_0$ of order $p$ under $\pi$ cannot be cyclic because of the maximality of $C'$. Hence,
\[
\pi^{-1}(T_0)=E[p]+C'
\]
and $T_0$ is the unique submodule of order $p$. Since $\pi$ restricted to $T$ is injective, we conclude that $T_0=\pi(T)$. Hence, $E_0[p]$ is a nontrivial extension of $\pdual{T}(1)$ by $T$. If $T$ is the only submodule of $E$ of order $p$, we may simply set $E_0=E$.

There exists a unique maximal cyclic $\Gal_K$-submodule $C$ of $E_0$ of order a power of $p$. Indeed, if $D\neq C$ is a second submodule with the same property, then $C+D$ cannot be cyclic. Hence, $C+D$ must contain $E_0[p]$ and $(C+D)^{\sesi}$ contains $\pdual{T}(1)$ as a submodule. But both $C$ and $D$ are cyclic and contain $T$ as a submodule. Hence, $C^{\sesi}$, $D^\sesi$ and $(C+D)^{\sesi}$ are powers of $T$. This contradicts our assumption that $T\neq \pdual{T}(1)$.

Consider $E_1=E_0/C$. Arguing as before, we see that $E_1$ contains a unique submodule $T_1$ of order $p$. Let $p^n$ be the order of $C$. The composition of the maps
\[
E_0[p]\subset E_0[p^n]\qto E_0[p^n]/C\subset E_1[p^n]
\]
shows that $T_1=E_0[p]/T_0\isomorph\pdual{T}(1)$. Hence, $E_1[p]$ is a nontrivial extension of $T$ by $\pdual{T}(1)$. Since $T$ is a submodule of $E_0[p]$, but not of $E_1[p]$, we have $E_0[p]\not\isomorph E_1[p]$.

Assume now that $E'[p]$ is not isomorphic to $E_0[p]$ or to $E_1[p]$. Because the group $\Hom_{\Gal_K}(E'[p],E_0[p])$ is nontrivial, we find a $\Gal_K$-equivariant homomorphism
\[
f\colon E'[p]\mto E_0[p]
\]
with kernel isomorphic to $\pdual{T}(1)$. Replacing $E_0$ by $E_1$, we see that $T$ is also a submodule of $E'[p]$. Hence, $E'[p]\isomorph T\oplus \pdual{T}(1)$.

If the order of $C$ is larger than $p$, then it is easy to check that $(E_0/T)[p]\isomorph T\oplus \pdual{T}(1)$. Conversely, assume that $E'$ has two distinct submodules
\[
T_1\isomorph T,\qquad T_2\isomorph \pdual{T}(1)
\]
of order $p$. Note that
\[
\Hom_{\Gal_K}(E_0[p],E'[p])=\Hom_{\Gal_K}(\pdual{T}(1),\pdual{T}(1))=\FF_p
\]
such that every element $E_0[p]\mto E'[p]$ may be lifted to an isogeny $E_0\mto E'$.
In particular, there exists an isogeny $\phi\colon E_0\mto E'$ whose restriction to $E_0[p]$ has kernel $T_0$. This implies that the $p$-primary component of $\ker \phi$ is cyclic and hence, it is contained in $C$. The $p$-primary component of $\phi^{-1}(T_1)$ is then also cyclic and contained in $C$. We conclude that $C$ must have an order of at least $p^2$.
\end{proof}

\begin{cor}\label{cor:Analysis of IKE}
Let $E/K$ be a fixed elliptic curve. If $E$ satisfies (a), (b), or (c) of Proposition \ref{prop:classification}, then Conjecture $I(K,E)$ holds.

If $E$ satisfies (d), then the following are equivalent:
\begin{enumerate}
\item\label{enum:IE1} Conjecture $I(K,E)$ holds.
\item\label{enum:IE2} $A(K,E_0[p])$ holds if and only if $A(K,E_1[p])$ holds.
\item\label{enum:IE3} Let $T\subset E_0[p]$ be the submodule of order $p$. Then $A(K,E_0[p])$ implies $A(K,T)$ and $A(K,E_1[p])$ implies $A(K,\pdual{T}(1))$.
\end{enumerate}
\end{cor}
\begin{proof}
If $E$ satisfies (a), then $I(K,E)$ is trivial, since $E[p]\isomorph E'[p]$ for every $E'$ isogenous to $E$. If $E$ satisfies (b), then for every $E'$ isogenous to $E$, $A(K,E')$ holds if and only if $A(K,T)$ holds.  If $E$ satisfies (c), then $E'[p]$ is semisimple for all $E'$ and $A(K,E')$ holds if and only if $A(K,T)$ and $A(K,\pdual{T}(1))$ holds. In both cases, either $A(K,E')$ holds for all $E'$ or for none.

Assume $E$ satisfies (d) and that $C\subset E_0$ is the unique maximal cyclic $\Gal_K$-submodule $C$ of $p$-power order. Clearly, \eqref{enum:IE1} implies \eqref{enum:IE2}. Assume \eqref{enum:IE2} and that $A(K,E_0[p])$ or $A(K,E_1[p])$ holds. Then the respective other holds as well. Since $T$ is a quotient of $E_1[p]$ and $\pdual{T}(1)$ a quotient of $E_0[p]$, we conclude from Lemma~\ref{lem:A under quotiens and extensions} that $A(K,T)$ and $A(K,E_1[p])$ hold. Hence, \eqref{enum:IE3} is true. Finally, assume \eqref{enum:IE3} and that $A(K,E'[p])$ holds for some $E'$ isogenous to $E$. If $E'[p]$ is semisimple, then both $T$ and $\pdual{T}(1)$ are quotients of $E'[p]$. Otherwise, $E'[p]$ is either isomorphic to $E_0[p]$ or to $E_1[p]$. In any case, we may deduce that $A(K,T)$ and $A(K,\pdual{T}(1))$ hold, either from \eqref{enum:IE3} or from Lemma~\ref{lem:A under quotiens and extensions}. Lemma~\ref{lem:A under quotiens and extensions} then also implies that $A(K,E')$ holds for every $E'$ isogenous to $E$. Hence, Conjecture $I(K,E)$ is true.
\end{proof}

\begin{rem}
In the case (d) of the above proposition, note that both $E_0[p]$ and $E_1[p]$ are quotients of the $\Gal_K$-representation $E[p^{n+1}]$ with $p^n$ denoting the order of the maximal cyclic $\Gal_K$-submodule $C$ of $p$-power order of $E_0$. Let $K(E[p^{n+1}])$ denote the minimal extension trivialising the Galois representation $E[p^{n+1}]$. Then both $E_0[p]$ and $E_1[p]$ may be viewed as specific extensions of one dimensional representations of $\Gal(K(E[p^{n+1}])/K)$ over $\FF_p$. In particular, the isogeny conjecture in this case boils down to understanding the relationship between these two extensions.
\end{rem}

Let us now discuss Conjecture $I(K,E)$ for CM elliptic curves. Let $\mathcal{O}$ be an order in an imaginary quadratic field $L$. Then $\mathcal{O}$ is of the form
\[
\mathcal{O}=\Int+f\mathcal{O}_L,
\]
with $f\in\N$ the conductor of $\mathcal{O}$ and $\mathcal{O}_L$ the maximal order of $L$. Recall that the discriminant of $\mathcal{O}$ is given by
\[
d_{\mathcal{O}}=f^2d_L,
\]
with $d_L$ denoting the discriminant of $L$. Recall that an elliptic curve $E/K$ is a CM elliptic curve if the endomorphism ring of $E$ over the algebraic closure $\algc{K}$ of $K$ is strictly larger than $\Int$. Let us first assume that $E/K$ is an elliptic curve which has complex multiplication by $\mathcal{O}$ over $K$ so that all endomorphisms over $\algc{K}$ are already defined over $K$. In this case, $E[p]$ is a free $\mathcal{O}/p\mathcal{O}$-module of rank 1 \cite[Lem.~1]{Parish:RationalTorsion} and the action of $\Gal_K$ on $E[p]$ is given by a character
\[
\psi\colon\Gal_K\mto(\mathcal{O}/p\mathcal{O})^\times.
\]

\begin{cor}\label{cor:CM over K}
Assume that $E/K$ has complex multiplication by $\mathcal{O}$ over $K$. Then Conjecture $I(K,E)$ holds.
\end{cor}
\begin{proof}
Write $g$ for the image of $\Gal_K$ in the automorphism group of $E[p]$. By Corollary~\ref{cor:Analysis of IKE} we may assume that $E[p]$ has a non-trivial $g$-subrepresentation $T$ of rank $1$ over $\FF_p$ such that $E[p]$ is a non-trivial extension of $\pdual{T}(1)$ by $T$ and such that $T\not\isomorph\pdual{T}(1)$. Consider the $g$-representation
\[
A=T\tensor_{\FF_p}T(-1)
\]
of rank $1$ over $\FF_p$. Then $E[p]$ represents a nontrivial class in
\[
\Ext^1_{g}(\pdual{T}(1),T)=\HF^1(g,A).
\]
If $E$ has complex multiplication over $K$, then $g$ is abelian. Let $h\subset g$ be the maximal subgroup of order prime to $p$. By our assumption that $T\not\isomorph\pdual{T}(1)$, $h$ acts non-trivially on $A$, so that $A^h=0$. The Hochschild-Serre spectral sequence then implies
\[
\HF^1(g,A)=\HF^0(g/h,\HF^1(h,A))=0
\]
in contradiction to our assumption that the class of $E[p]$ is nontrivial.
\end{proof}

Let us now consider the case when $E/K$ does not have complex multiplication over $K$, but does so over the larger field $KL$.

\begin{lem}\label{lem:non-trivial extension in CM case}
Let $E/K$ be an elliptic curve such that
\begin{enumerate}
\item[(i)] $E$ has complex multiplication by $\mathcal{O}$ over $KL\neq K$,
\item[(ii)] there exists a $\Gal_K$-submodule $T\subset E[p]$ of rank $1$ over $\FF_p$,
\item[(iii)] $T\not\isomorph \pdual{T}(1)$ as $\Gal_K$-modules,
\item[(iv)] $E[p]$ is a non-trivial extension of $\pdual{T}(1)$ by $T$.
\end{enumerate}
Then $p$ divides the discriminant $d_{\mathcal{O}}$ and $T\isomorph \pdual{T}(1)$ as $\Gal_{LK}$-modules. If $p\equiv 3 \mod 4$ and $\sqrt{-p}\in K$, then such a curve cannot exist.
\end{lem}
\begin{proof}
We write $g'$ for the image of $\Gal_K$ in the automorphism group of $E[p]$ and set $A=T\tensor_{\FF_p}T(-1)$ as before. Then $g'$ is an extension of the group
\[
z=\Gal(KL/K)
\]
of order $2$ by an abelian group $g$. Since we assume $p\neq 2$, the restriction map
\[
\HF^1(g',A)\mto \HF^1(g,A)
\]
is injective. Since $E[p]$ represents a non-trivial class in $\HF^1(g',A)$, we conclude that $\HF^1(g,A)\neq 0$. By the same argument as in Corollary~\ref{cor:CM over K}, this can only be true if the action of $g$ on $A$ is trivial. In particular,  $T$ and $\pdual{T}(1)$  are isomorphic as $\Gal_{LK}$-modules. Moreover, $g$ must contain a non-trivial subgroup of order $p$. This can only happen if $p$ divides the discriminant $d_{\mathcal{O}}$. Indeed, if this were not the case, then the primes of $\mathcal{O}$ above $p$ are regular in the sense of \cite[\S~12]{Neukirch:ANT} and unramified in $L/\Rat$. If $p$ is inert in $L$, then $E[p]$ is a vector space of dimension $1$ over the residue field $k$ of $p$ in $L$ with an faithful action of $g'$ by $k$-linear automorphisms. In particular, the order of $g'$ is prime to $p$. If $p$ splits in $L$, then the restriction of $E[p]$ to $\Gal_{LK}$ is split. In both cases, $E[p]$ represents the trivial class in $\HF^1(g',A)$, in contradiction to our assumptions.

Now, assume that $p\equiv 3\mod 4$ and that $\sqrt{-p}\in K$. Let $\sigma\in g$ be an element which is not contained in $g'$. Since $\frac{p-1}{2}$ is an odd integer, the element $\sigma'=\sigma^{\frac{p-1}{2}}$ is also not contained in $g'$. However, since $\sqrt{-p}\in K$, the element $\sigma'$ fixes $K(\mu_p)$. Moreover, $\sigma'$ acts trivially on $T\tensor_{\FF_p} T$. Hence, $\sigma'$ acts trivially on $A$, so that $A$ is in fact a trivial $g'$-module, in contradiction to assumption (iii).
\end{proof}

\begin{cor}
Assume that $p\equiv 3 \mod 4$ and that $\sqrt{-p}\in K$. Then $I(K,E)$ holds for all CM elliptic curves $E/K$.
\end{cor}
\begin{proof}
Let $E/K$ be any CM elliptic curve. By Corollary~\ref{cor:CM over K} we may assume that $E$ does not have CM over $K$. By Lemma~\ref{lem:non-trivial extension in CM case}, case (d) of Proposition~\ref{prop:classification} can never occur. Hence, $I(K,E)$ holds by Corollary~\ref{cor:Analysis of IKE}.
\end{proof}

\section{Closing Remarks}\label{sec:ClosingRemarks}

In the light of \eqref{enum:IE3} of Corollary~\ref{cor:Analysis of IKE}, we may formulate the following conjecture for arbitrary $\Gal_K$-representations $M$ on finite-dimensional vector spaces over $\FF_p$.

\begin{conj}[$\Sub(K,M)$]
For every subrepresentation $N\subset M$, $A(K,M)$ implies $A(K,N)$.
\end{conj}

In particular, for an elliptic curve $E/K$, $I(K,E)$ is equivalent to the conjunction of $\Sub(K,E'[p])$ for every $E'$ which is $K$-isogenous to $E$. Moreover, note that $\Sub(K,M)$ holds if every subrepresentation of $M$ also appears as a quotient representation of $M$. In particular, $\Sub(K,M)$ holds if the semisimplification of $M$ is a power of a simple $\Gal_K$-module. This can always be achieved by passing to a finite extension $K'$ of $K$. We also note the following.

\begin{lem}\label{lem:descent of S under p-extensions}
Let $L/K$ be a Galois extension of order a power of $p$ and $M$ be a $\Gal_K$-representations on a finite-dimensional vector space over $\FF_p$. Then $\Sub(L,M)$ implies $\Sub(K,M)$.
\end{lem}
\begin{proof}
This is an immediate consequence of Lemma~\ref{lem:A under field extensions}.(4).
\end{proof}

However, beware that in general, $\Sub(K,M)$ does not imply $\Sub(L,M)$, since the set of $\Gal_L$-subrepresentations of $M$ might be strictly larger than the set of $\Gal_K$-subrepresentations.

\begin{cor}
Let $M$ be a $\Gal_K$-representations on a finite-dimensional vector space over $\FF_p$ unramified outside of $S$. Then there exists a finite subextension $K'/K$ of $K_S/K$ of degree prime to $p$ such that $\Sub(K',M)$ holds. Likewise, if $E/K$ is an elliptic curve with good reduction outside of $S$, then there exists a finite subextension $K'/K$ of $K_S/K$ of degree prime to $p$ such that $I(K',E)$ holds.
\end{cor}
\begin{proof}
Choose $L/K$ to be a Galois subextension of $K_S/K$ trivialising the $\Gal_K$-representations $M$ and $E[p]$, respectively. In particular, $E/L$ satisfies case (b) of Proposition~\ref{prop:classification}, so that $I(L,E)$ holds by Corollary~\ref{cor:Analysis of IKE}. Conjecture $\Sub(L,M)$ is obviously true.

Now choose $K'\subset L$ to be the fixed field of a $p$-Sylow subgroup of $\Gal(L/K)$. Then $\Sub(K',M)$ and $I(K',E)$ hold by Lemma~\ref{lem:descent of S under p-extensions}.
\end{proof}

If $L/K$ is Galois of order prime to $p$, it is in general not possible to infer $\Sub(K,M)$ from $\Sub(L,M)$. To overcome this deficiency, it might be worthwhile to study the following strengthening of $\Sub(K,M)$. As before, we let $S$ denote a finite set of primes of $K$ containing the primes above $p$ and those primes where $M$ is ramified.

\begin{conj}[$C(K,M)$]
For every $\Gal_K$-subrepresentation $N\subset M$,
\[
Q\tensor_{\Omega}\HF^2(\Gal_S(K),N_{K_\cyc})\mto Q\tensor_{\Omega}\HF^2(\Gal_S(K),M_{K_\cyc})
\]
is injective.
\end{conj}

\begin{lem}
Let $M$ be a $\Gal_K$-representation which is unramified outside of $S$. Then the following are equivalent:
\begin{enumerate}
\item Conjecture $C(K,M)$ holds,
\item for every $\Gal_K$-subrepresentation $N\subset M$,
\[
\ker(\HF^2(\Gal_S(K_\cyc),N)\mto\HF^2(\Gal_S(K_\cyc),M))
\]
is finite,
\item for every $\Gal_K$-subrepresentation $N\subset M$,
\[
\ker(\HF^2(\Gal_S(K),N_{K_\cyc})\mto \HF^2(\Gal_S(K),M_{K_\cyc}))
\]
is finite,
\item for every $\Gal_K$-subrepresentation $N\subset M$,
\[
Q\tensor_{\Omega}\HF^1(\Gal_S(K),M_{K_\cyc})\mto Q\tensor_{\Omega}\HF^1(\Gal_S(K),(M/N)_{K_\cyc})
\]
is surjective,
\item for every $\Gal_K$-subrepresentation $N\subset M$,
\[
\coker(\HF^1(\Gal_S(K_\cyc),M)\mto\HF^2(\Gal_S(K_\cyc),M/N))
\]
is finite,
\item for every $\Gal_K$-subrepresentation $N\subset M$,
\[
\coker(\HF^1(\Gal_S(K),M_{K_\cyc})\mto \HF^1(\Gal_S(K),(M/N)_{K_\cyc}))
\]
is finite,
\item Conjecture $C(K,\pdual{M}(1))$ holds.
\end{enumerate}
\end{lem}
\begin{proof}
The equivalences (1)--(6) follow easily from the exact sequence in Lemma~\ref{lem:coefficient sequence}. To show that (7) implies (2), we compare the Poitou-Tate sequences~\eqref{eqn:Poitou-Tate-II} for $M$ and $N$. The same argument with $M$ replaced with $\pdual{M}(1)$ together with the equivalence of (1) and (2) applied to $\pdual{M}(1)$ shows that (1) implies (7).
\end{proof}

\begin{lem}
Conjecture $C(K,M)$ does not depend on the finite set of primes $S$.
\end{lem}
\begin{proof}
Let $T$ be a finite set of prime containing $S$ and note that $\HF^0(\Gal_{k(v)},M(-1))$ is a finite group for every $v\in (T-S)_{K_\cyc}$. The claim of the lemma then follows from the Gysin sequence \eqref{eqn:Gysin sequence}.
\end{proof}

\begin{lem}\label{lem:descent of C}
Let $L/K$ be a finite extension of degree prime to $p$ and $M$ a $\Gal_K$-representation on a finite-dimensional vector space over $\FF_p$. Then $C(L,M)$ implies $C(K,M)$.
\end{lem}
\begin{proof}
The restriction map $\HF^2(\Gal_S(K_\cyc),M)\mto \HF^2(\Gal_S(L_\cyc),M)$ is split injective, as its composition with the corestriction map is the multiplication by the degree of $L/K$. Hence, if $N\subset M$ is a $\Gal_K$-subrepresentation such that the group
$\ker(\HF^2(\Gal_S(L_\cyc),N)\mto\HF^2(\Gal_S(L_\cyc),M))$ is finite, then $\ker(\HF^2(\Gal_S(K_\cyc),N)\mto\HF^2(\Gal_S(K_\cyc),M))$ is also finite.
\end{proof}

Conjecture $C(K,M)$ is connected with the isogeny conjecture $I(K,E)$ for elliptic curves in the case when $M$ is a $2$-dimensional reducible representation over $\FF_p$.  We then have an exact sequence
\[
 0\mto N\mto M\mto N'\mto 0
\]
with one-dimensional representations $N$ and $N'$. Passing to an extension of degree prime to $p$ and using Lemma~\ref{lem:descent of C}, we can further reduce to the case that $N$ and $N'$ are trivial representations. In this case, $M$ represents a class $\xi$ in $\HF^1(\Gal_S(K),\FF_p)$ and the connecting homomorphism
\[
Q\tensor_{\Omega}\HF^1(\Gal_S(K),(\FF_p)_{K_\cyc})\mto Q\tensor_{\Omega}\HF^2(\Gal_S(K),(\FF_p)_{K_\cyc})
\]
is given by the cup product with $\xi$. So, in essence, Conjecture $C(K,M)$ for reducible two-dimensional representations $M$ boils down to the vanishing of cup products. We plan to investigate this further in subsequent work.

\bibliographystyle{amsalpha}
\bibliography{Literature}

\end{document}